\documentclass[12pt, reqno]{amsart}
\usepackage[centertags]{amsmath}
\usepackage{amsfonts}
\usepackage{amssymb}
\usepackage[latin5]{inputenc}
\usepackage{amsthm}
\usepackage{mathrsfs}
\usepackage{upgreek}
\usepackage{amsopn}
\usepackage{amscd}
\newtheorem{theorem}{Theorem}[section]
\newtheorem{lemma}[theorem]{Lemma}
\newtheorem{definition}[theorem]{Definition}
\newtheorem{cor}[theorem]{Corollary}

\theoremstyle{definition}

\theoremstyle{remark}

\numberwithin{equation}{section}
\newcommand{\ou}{\textnormal{ou}\, }
\newcommand{\ic}{\textnormal{int}\, }
\renewcommand{\leq}{\leqslant}
\renewcommand{\geq}{\geqslant}

\renewcommand{\emptyset}{\varnothing}
\renewcommand{\epsilon}{\varepsilon}
\newcommand{\included}{\subseteq}      

\newcommand{\stnd}[1]{\mathbb{#1}}
\newcommand{\N}{\stnd{N}}         
\newcommand{\R}{\stnd{R}}         

\begin{document}
{\setlength{\baselineskip}%
                        {1.3\baselineskip}
                      
\title{Order-Unit-Metric Spaces}
\author{mert \c{c}a\u{g}lar}
\address{(mert \c{c}a\u{g}lar) Department of Mathematics and Computer Science, \.{I}stanbul K\"{u}lt\"{u}r University, Bak\i rk\"{o}y 34156,
 \.{I}stanbul, Turkey}
\email{m.caglar@iku.edu.tr}

\author{Zafer Ercan}
\address{(zafer ercan) Department of Mathematics, Abant \.{I}zzet Baysal University, G\"{o}lk\"{o}y Kamp\"{u}s\"{u}, 14280 Bolu, Turkey}
\email{zercan@ibu.edu.tr}


\maketitle 
\begin{abstract} We study the concept of cone metric space in the context of ordered vector spaces by setting up a general and natural framework for it.
    
\end{abstract}
\section{Introduction} 
\noindent 
It is well-known that on vector spaces, vector orderings and cones are in one-to-one correspondence. Moreover, ordered Hausdorff topological vector spaces with a non-empty interior are precisely the ones with a closed cone. This is one of the main motivations in the study of cone metric spaces (see, \cite{aj, HX, dr, rh}). The purpose of this note is to show that a typical generalization of the concept of cone metric space, which we call order-unit-metric space, originates from and naturally fits into the framework of the rich theory of ordered vector spaces (see, \cite{at, ja, pe}). We will set up a general and natural framework which contains, in particular, the cone metric spaces. For the standard terminology and all unexplained notation on ordered vector spaces, we refer to \cite{at}.


The following defines the central object under our consideration and is given in \cite{HX}, where $\ic (K)$ denotes the interior of the set $K$.
\begin{definition}\label{def1} Let $X$ be a non-empty set and $E$ be an ordered real Banach space with $\ic(K)\neq\emptyset$, where $K$ is a closed cone in $E$. A function $d:X\times X\rightarrow E$ is called a {\bf\emph{cone metric}} if:
\begin{enumerate} 
\item[\textnormal{(i)}] $d(x,y)=0$ if and only if $x=y$;
\item[\textnormal{(ii)}] $d(x,y)=d(y,x)$ for all $x$, $y\in X$;
\item[\textnormal{(iii)}] $d(x,y)\leq d(x,z)+d(y,z)$ for all $x$, $y$, $z\in X$.
\end{enumerate}
In this case, the quadruple $(X, E, K, d)$ is called a {\bf\emph{cone metric space}}.
\end{definition}
The ordering ``$\leq$" of the Banach space $E$ in the above definition is generated by the cone $K$ in the usual way. It is also clear that replacing the ordered Banach space $E$ by an ordered topological vector space, the notion of cone metric space can further be generalized: albeit doable, we will not deal with this generalization but take a different path instead and analyze the topology of order-unit-metric spaces.   

As usual, the open ball of a metric space $(X, d)$ with center $x\in X$ and radius $r>0$ is denoted by $B_d(x, r)$ or, for short, by $B(x, r)$ if there is no ambiguity of confusion. Let $(X, E, K, d)$ be a cone metric space. Khami and Pourmahdian \cite{kp} has noted that the family $\{B_{\ll}(x,r): x\in X,\ 0\ll r\}$, where $B_{\ll}(x, r):=\{y\in X : d(x,y)\ll r\}$, forms a basis for a topology on $X$, called the {\bf\emph{cone metric topology}} of the corresponding cone metric space. Here $x\ll y$ stands for the relation $y-x\in\ic (K)$. One of the main issues of the present note is to prove, in a more general setting (cf. Corollary \ref{cor1} below), that the cone metric topology is metrizable. 
\section{Order-Unit-Metric Spaces} 
\noindent We will introduce in this section spaces that extend cone metric spaces naturally. We first need some preparation to study the underlying concepts. 

The notion of an order-unit of an ordered vector space, first given in \cite{kadison}, is as follows (see also \cite[p. 5]{at}, \cite[p. 11]{ja}, and \cite[p. 4]{pe}).
\begin{definition}\label{def2} Let $E$ be an ordered real vector space. An element $0<e\in E$ is called an {\bf\emph{order-unit}} if for each $x\in E$ there exists a $\lambda>0$ such that $x\leq\lambda e$. The set of order units of $E$ will be denoted by $\ou(E)$.
\end{definition}
\noindent Note that the inclusions
$\ou(E)+\ou(E)\included \ou(E)$ and $(0, \infty )\, \ou(E)\included\ou(E)$ hold. Recall that an ordered vector space $E$ is called {\bf\emph{Archimedean}} whenever it follows from $y\in E$, $x\in E^+$, and $\N y\leq x$ that $y\leq 0$. The characterization of the vectors belonging to the interior of the cone of an ordered topological vector space is crucial for our considerations and is given in the following well-known result. A proof of this can be found in \cite[Lemma 2.5]{at} or, in the more general setting of wedges, in \cite[1.3.1]{ja}; we take it up here for the sake of completeness in a relevant form regarding our purposes. In what follows, the notation $(E, K)$ will suggest that the vector ordering making the vector space $E$ into an ordered vector space is generated by the cone $K$ in $E$.
\begin{theorem}\label{thm1} Let $(E, K, \tau)$ be an ordered topological vector space; that is, $(E, K)$ is an ordered vector space and $(E, \tau)$ is a topological vector space. Then, the following are equivalent:
\begin{enumerate} 
\item[\textnormal{(i)}] $e\in \textnormal{int}(K)$.
\item[\textnormal{(ii)}] $[-e,e]$ is a neighborhood of zero.
\end{enumerate} 
In this case, $e$ is an order unit.
\end{theorem}
\begin{proof} If $[-e,e]$ is a neighborhood of zero, then since 
$$e+[-e,e]=[0,2e]\included K,$$ we have $e\in \textnormal{int}(K)$. If $e\in \textnormal{int}(K)$, then there is a circled neighborhood of zero $V$ with $e+V\included K$. Then $V\included [-e,e]$: indeed, for each $x\in V$ we have $-x\in V$, so $e+x$ and $e-x$ both lie in $K$. This shows that $x\in [-e,e]$, whence $[-e,e]$ is a neighborhood of zero.
\end{proof}
To put our considerations into perspective, we need to recall two pertinent results: the former is Lemma 2.4 and the latter is Theorem 2.8, of \cite{at}.
\begin{theorem}\label{thm2} Let $(E, K, {\tau})$ be an ordered Hausdorff topological vector space for which $\ic(K)\neq\emptyset$. Then the cone $K$ is Archimedean if and only if $K$ is closed.
\end{theorem}
\begin{theorem}\label{thm3} Let $(E, K, {\tau})$ be an Archimedean ordered Hausdorff topological vector space with $\ic (K)\neq\emptyset$, and $(E, {\tau})$ be completely metrizable. Then $\ic (K)=\ou (K)$.
\end{theorem}

We are now ready to introduce the natural generalization of the notion of cone metric space. 

\begin{definition}\label{def2} Let $X$ be a non-empty set and $E$ be an ordered real vector space with $\ou(E)\neq\emptyset$. A function $d:X\times X\rightarrow E$ is called an {\bf\emph{order-unit-metric}} if:
\begin{enumerate} 
\item[\textnormal{(i)}] $d(x,y)=0$ if and only if $x=y$;
\item[\textnormal{(ii)}] $d(x,y)=d(y,x)$ for all $x$, $y\in X$;
\item[\textnormal{(iii)}] $d(x,y)\leq d(x,z)+d(y,z)$ for all $x$, $y$, $z\in X$.
\end{enumerate}
In this case, the triple $(X, E, d)$ is called an {\bf\emph{order-unit-metric space}}.
\end{definition}
Note that $d$ takes values in the cone $E_+:=\{e\in E\mid e\geq 0\}$ of $E$; that is, $d(x,y)\geq 0$ for each $x, y\in X$. Moreover, it is clear that every cone metric space is an order-unit-metric space, since $\ic (E_+)\included\ou(E)$ by Theorem \ref{thm1}. It should also be pointed out, comparing with Definition \ref{def1}, that in case $E$ is a Banach space in Definition \ref{def2} more is true: recall that an ordered Banach space having a closed cone and order units is called a {\bf\emph{Krein space}} \cite[Definition 2.62]{at}, and for such a space $E$, the cone $E_+$ is Archimedean by Theorem \ref{thm2}, and the order units are precisely the interior points of $E_+$ by Theorem \ref{thm3}. In particular, an ordered Banach space with a closed cone is a Krein space if and only if the interior of its cone is non-empty: ordered Banach spaces $E$ appearing in the definition of cone metric spaces are, in fact, Krein spaces. That the notions of order-unit-metric space and cone metric space coincide whenever the space $E$ is completely metrizable will be deferred until Theorem \ref{thm9}.
\section{Order-Unit-Topological Spaces}
\noindent The subject matter of this section is the topology of order-unit-metric spaces and convergence concepts therein. We will write $x\lll y$ if $y-x\in\ou(E)$, where $(X, E, d)$ is an order-unit-metric space. 

The proof of the following is elementary and is therefore omitted.
\begin{theorem}\label{thm4} The family $\{B_{\lll}(x,r): x\in X,\ 0\lll r\}$ for an order-unit-metric space $(X, E, d)$, where $B_{\lll}(x, r):=\{y\in X : d(x,y)\lll r\}$,  forms a basis for a topology on $X$. 
\end{theorem}
We shall call this topology as the {\bf\emph{order-unit-topology}} of the corresponding order-unit-metric space. The following observation establishes the fact that the cone metric topology can be recovered via the order-unit-topology.
\begin{theorem}\label{thm5} Let $(X, E, d)$ be an order-unit-metric space and $e\in\ou(E)$ be fixed. Then the family $\{B_{\lll}(x, re) : 0<r\in\mathbb{R}^{+}\}$ is a basis for the cone metric topology.
\end{theorem}
\begin{proof} Let $u\in\ic (E_+)$, which is an order-unit, be given. Then $re\leq u$ for some $r>0$, from which it directly follows that $B_{\ll}(x,re)\included B_{\lll}(x,u)$, whence the proof.
\end{proof}
We will proceed by recalling some terminology regarding convergence properties of sequences in cone metric spaces, which will be used in the sequel. 

Let  $(X,E,K,d)$ be a cone metric space and  $(x_{n})$ be a sequence in $X$. The sequence $(x_{n})$ is said to be $\ll$-$u$-{\bf\emph{Cauchy}}, where $0\ll u$, if for every ${\epsilon}>0$ there exists a $k\in\N$ such that $d(x_{n},x_{m})\ll\epsilon u$ for all $n,m\geq k$. The sequence $(x_{n})$ is called $\ll$-{\bf\emph{Cauchy}} if it is $\ll$-$u$-Cauchy for all $u\in E$ with $0\ll u$.   The sequence $(x_{n})$ is said to be $\ll$-$u$-{\bf\emph{convergent}} to $x\in X$, where $0\ll u$, if for every ${\epsilon}>0$ there exists a $k\in\N$ such that $d(x_n,x)\ll\epsilon u$ for all $n,m\geq k$; this is denoted by $x_n\to x\, (\ll u)$. The sequence $(x_n)$ is $\ll$-{\bf\emph{convergent}} to $x\in X$, denoted by $x_n\rightarrow x\, (\ll)$, if $x_{n}\to x\, (\ll u)$ for all $u\in E$ with $0\ll u$. 

For a fixed $e\in E$ with $0\ll e$, the following facts are immediate:
\begin{equation}\begin{aligned}\label{eq0}
\text{(i)}\ & \text{$(x_n)$ is $\ll$-Cauchy if and only if it is $\ll$-$e$-Cauchy.}\\
\text{(ii)}\ & \text{$x_n\to x\, (\ll )$ if and only if $x_{n}\to x\, (\ll e)$.}
\end{aligned}\end{equation}
One should note that replacing the relation ``$\ll$'' (defined in cone metric spaces) in the above definitions by the relation ``$\lll$'', the corresponding concepts can be defined verbatim in order-unit-metric spaces. As an instance, if $(X, E, d)$ is an order-unit-metric space, then then a sequence $(x_n)$ in $X$ is said to be $\lll$-$u$-{\bf\emph{Cauchy}}, where $0\lll u$, if for every $\epsilon>0$ there exists a $k\in \N$ such that $d(x_{n},x_{m})\lll\epsilon u$ for all $n,m\geq k$. The proof of the following result, revealing the interdependence of these concepts, requires no effort and is therefore not given.
\begin{theorem}\label{thm6} Let $(X,E,d)$ be an order-unit-metric space, $e\in\ou(E)$, $x\in X$, and $(x_n)$ be a sequence in $X$. Then, the following are equivalent:
\begin{enumerate} 
\item[\textnormal{(i)}] $x_n\rightarrow x$ in the order-unit-topology. 
\item[\textnormal{(ii)}]  $x_n\rightarrow x\, (\lll e)$.   
\item[\textnormal{(iii)}] $x_n\rightarrow x\, (\lll)$.
\end{enumerate}
\end{theorem}
\section{Metrizability of Order-Unit-Topological Spaces}

\noindent The main results of the present note concerning the metrizability of order-unit-topology-\ cal spaces along with their consequences will be presented in this section. The principal machinery we use originates from the works of R. Kadison and F.F. Bonsall concerning representations of topological algebras and ideals in partially ordered spaces, respectively.

Let $E$ be an Archimedean ordered real vector space with an order unit $e>0$. We denote the (necessarily non-empty) set of positive functionals $f:E\rightarrow\mathbb{R}$ with $f(e)=1$ by $H_{e}$. As remarked by Kadison \cite[p. 405]{kadison}, the space $H_{e}$ with the weak*-topology is a compact Hausdorff space. The basic result in this direction is the following, due to Bonsall \cite[Theorem 4 \& Corollary 2]{b}.
\begin{theorem}\label{thm7} Let $E$ be an Archimedean ordered vector space with an order unit $e>0$. Then $E$ is isomorphic to a subspace of $C(H_{e})$ via the map
$${\pi}:E\rightarrow C(H_e),\quad {\pi}(x)(f):=f(x).$$
\end{theorem}
If $E$ is as in Theorem \ref{thm7}, then, obviously, the map  
\begin{equation}\label{eq1} p:E\to\R,\quad p(x):=\|\pi(x)\|\end{equation} is  a norm, where $\|\cdot\|$ is the usual supremum norm on $C(H_e)$. Before presenting the main result, let us observe the following fact which will be used in the proof of it. 

\begin{lemma}\label{new} Let $E$ be an ordered vector space, $e\in E$ be an order unit, and $x\in E$. Then, the following are equivalent:
\begin{enumerate} 
\item[\textnormal{(i)}] $x<re$ for some $0<r<1$. 
\item[\textnormal{(ii)}]  $e-x$ is an order unit of $E$.   
\end{enumerate}
\end{lemma}
\begin{proof} Suppose that (i) holds. Choose $\lambda>1$ with $r\lambda=\lambda-1$. Then $e<\lambda(e-x)$, which shows that $e-x$ is an oder unit. Now suppose that $e-x$ is an order unit. Then there exists $0<{\epsilon}<1$ such that ${\epsilon}e<e-x$, whence $x<(1-\epsilon)e$. Choose now $r:=1-\epsilon$. 
\end{proof}
We are now in a position to give the main result of this note.
\begin{theorem}\label{thm8} Let $(X, E, d)$ be an Archimedean order-unit-metric space and $e\in\ou (K)$ be fixed. Then the order-unit-topology of $(X, E, d)$ is metrizable.
\end{theorem}
\begin{proof} Define the map $\overline{d}:X\times X\to\R$ by 
\begin{equation}\label{eq2}\overline{d}(x, y):=p(d(x,y)),\end{equation} where $p$ is the norm given in (\ref{eq1}). Clearly, $p$ is a metric on $X$. We claim that the topology of the metric space $(X,\overline{d})$ and the order-unit-topology coincide. The proof of this will be completed upon considering the following chain of four claims, where we keep $x\in X$ as fixed.

{\it Claim 1.}  One has 
$$A:=B_{\lll}(x, e)=B:=\bigcup_{0<r<1}\{y : d(x,y) < re\}=C:=\bigcup_{0<r<1}B_{\overline{d}}(x, r).$$

Indeed, the equality of $A=B$ follows from Lemma \ref{new}. Let $y\in B$. Then $d(x,y) < re$ for some $0<r<1$. Choose $r<r'<1$. Then $\overline{d}(x,y)\leq r < r'$, so $y\in C$, therewith $B\included C$. It is obvious that $C\included B$. Thus $A=B=C$, proving the claim.

{\it Claim 2.} For every $u\in E$ with $0\lll u$, the set $B_{\lll}(x, u)$ is $\overline{d}$-open. 

Indeed, from Claim 1, this is true for $e=u$. Let $\mu > 0$ be given. As in Claim 1, we have 
$$B_{\lll}(x, \mu e)=\bigcup_{0<r<1}\{y : d(x,y) < r \mu e\}={1\over \mu} \bigcup_{0<r<1}B_{\overline{d}}(x, r),$$ showing that the claim is also true for $u=\mu e$.

{\it Claim 3.} For each $0<r\in\mathbb{R}$, there exists $k\in\mathbb{N}$ such that 
$$B_{\lll}(x, (1/k)\, e)\included B_{\overline{d}}(x,r).$$
Seeking for a contradiction, suppose that the claim is not true. Then there exists a sequence $(y_n)$ in $X$ and a sequence $(\lambda_n)$ in $\mathbb{R}^{+}$ with ${\lambda}_{n}\rightarrow {\infty}$ such that 
$$e\leq\lambda_n\left({1\over n}\, e-d(x, y_n)\right)\quad\textnormal{and}\quad r\leq \overline{d}(x, y_n)$$ for all $n$. (For the first inequality, the fact that ${1\over n}\, e-d(x,y_n)$ is an order unit of $E$ for all $n\in\mathbb{N}$ is used.) This implies 
$\lambda_n\, d(x, y_n)\leq(\lambda_n/n-1)\, e,$ from which it follows that
$$0<r\leq\overline{d}(x, y_n)\leq\left({1\over n}-{1\over\lambda_n}\right)\to 0,$$ a contradiction.

{\it Claim 4.} Every $\overline{d}$-open ball is open in cone-metric-topology: this follows immediately from Claim 3. 

The proof of the theorem is now complete.
\end{proof}
Let us call an order-unit-metric space $(X, E, d)$ {\bf\emph{complete}} whenever the metric space $(X,\overline{d})$, with the metric $\overline{d}$ given in (\ref{eq2}), is complete. Note that the statement of Theorem \ref{thm8} contains implicitly the fact that different metrics given by different order-units are equivalent. Order-unit-metric spaces that are complete can be characterized as follows.
\begin{theorem}\label{thm9} Let $(X, E, d)$ be an order-unit-metric space. Then, the following are equivalent:
\begin{enumerate} 
\item[\textnormal{(i)}] $(X, E, d)$ is complete.
\item[\textnormal{(ii)}] If $(x_{n})$ is $\lll$-Cauchy then there exists $x\in X$ such that $(x_{n})\rightarrow x(\lll )$.
\end{enumerate} 
\end{theorem}
\begin{proof} We will only prove that (i) implies (ii) as the proof of the reverse implication is quite similar; this, though, will be established by showing that $\lll$-Cauchy sequences in  $(X, E, d)$ and Cauchy sequences in $(X,\overline{d})$ do coincide. 

Let $(X, E, d)$ be a complete order-unit-metric space. Suppose that $(x_n)$ is $\lll$-Cauchy. Let $\epsilon > 0$ be given. Choose $k\in\mathbb{N}$ such that $d(x_n, x_m)\lll\epsilon e$ for all $n, m\geq k.$ Now for all $n, m\geq k$, since $\epsilon e-d(x_n, x_m)\in\ou (E)$, there exists a sequence $\lambda_{n, m}$ such that $e < \lambda_{n,m}(\epsilon e-d(x_n, x_m))$, which implies that 
$d(x_n, x_m)<\epsilon-1/\lambda_{n,m}<\epsilon$.  Hence $(x_{n})$ is Cauchy in $(X,\overline{d})$. 

Now suppose that  $(x_n)$ is Cauchy in $(X,\overline{d})$. Choose $k\in\mathbb{N}$ such that 
$$\overline{d}(x_n, x_m)<\epsilon^2={\mu-1\over\mu}\cdot\epsilon$$ for all $n, m\geq k$, where $\epsilon=(\mu-1)/\mu$. Now for $n, m\geq k$ , one has 
$$d(x_n, x_m)<\epsilon^2e={\mu-1\over\mu}(\epsilon e),$$ which implies 
$\epsilon e\leq\mu(\epsilon e-d(x_n, x_m))$. Since $e$ is an order unit, $\epsilon e-d(x_n, x_m)$ is an order unit as well, so $d(x_n, x_m)\lll\epsilon e$ for all $n, m\geq n_0$ for some $n_0\in\mathbb{N}$. Hence $(x_n)$ is $\lll$-$e$-Cauchy, and by (\ref{eq0}), it follows that $(x_n)$ is $\lll$-Cauchy, too.
\end{proof}
Let us call a cone metric space $(X, E, K, d)$ $\ll$-{\bf\emph{complete}} if every $\ll$-Cauchy sequence is $\ll$-convergent. With this in hand, the main result of \cite{kp} can be recovered in a more general setting as follows.
\begin{cor}\label{cor1} Every cone metric space $(X, E, K, d)$ is metrizable. In addition, if $(X, E, d)$ is complete, then $(X, E, K, d)$ is $\ll$-complete. 
\end{cor}
\begin{proof} Since $E$ is a Banach space it is completely metrizable, so the cone metric space $(X, E, K, d)$ and the corresponding order-unit-metric space coincide by Theorem \ref{thm8}; so, by Theorem \ref{thm9}, $(X, E, K, d)$ is $\ll$-complete.
\end{proof}
\section{Remarks on Some Applications}
\noindent Finally, we will comment on some applications of cone metric spaces in the literature in the light of our considerations, where the metrizability of the order-unit-topology simplifies the corresponding results. 

\noindent{\bf (1)} One of the areas where cone metric spaces is mostly used is the fixed point theory, where contraction mappings are the key ingredient. Recall that a mapping $f : X\to X$, where $(X, d)$ is a metric space, is called a {\bf\emph{contraction}} if there exists $0<c<1$ such that 
$$d(f(x),f(y))\leq c\, d(x, y)$$ for all $x, y\in X$. It is readily seen that if $f : X\to X$ is a contraction for a cone metric space $(X, E, K, d)$, then $f$ is also a contraction for the metric space $(X, \overline{d})$ with the metric $\overline{d}$ in (\ref{eq2}). Relying on this observation, some basic facts that are true for metric spaces and are shown to be true for cone metric spaces with considerable effort can be directly obtained with the ultimate ease of order-unit-metric spaces. As an instance, the following is the main result of \cite{rh}: 
\begin{quote}
\emph{Let $(X, E, K, d)$ be a complete cone metric space and $T:X\to X$ be contractive with 
$$d(T(x),T(y))\leq k\, d(x,y)$$ for all $x, y\in X$, where $0\leq k<1$. Then, $T$ has a unique fixed point in $X$, and for each $x\in X$, the iterative sequence $(T^{n}(x))$ converges to the fixed point.}
\end{quote}

\noindent This is also the main result of \cite{HX} for $K$ normal, where a cone $K$ is called {\bf\emph{normal}} if there exists $k>0$ such that $0\leq x\leq y$ implies $\|x\|\leq k\,\|y\|$. Note that cone metric spaces with a normal cone are Archimedean since the ordered Banach space in Definition \ref{def1} admits an equivalent monotone norm; similarly, ordered Banach spaces with a closed cone are Archimedean as well$-$see, e.g., \cite[Theorem 2.38]{at}. In this case, using the metric $\overline{d}$, we have 
$$\overline{d}(T(x),T(y))\leq k\, \overline{d}(x,y)$$ for all $x, y\in X$. Thus, the above result is true for $E=\R$, which shows that the fixed point theorems concerning cone metric spaces follow from the usual fixed point theorem on complete metric spaces. \\

\noindent{\bf (2)} Observe that if 
$$d(a,b)\leq \sum_{i=1}^{n}d(a_{i},b_{i}),$$ where $(X, E, K, d)$ is a cone metric space, then clearly 
$$\overline{d}(a,b)\leq \sum_{i=1}^{n}\overline{d}(a_{i},b_{i})$$ holds. This implies that the corresponding results of the paper \cite{aj} can be shown to hold by taking the ordered Banach space $E=\R$ with $K=[0,{\infty})$; moreover, these results do not depend on the particular cone metric space.\\

\noindent{\bf (3)} The results of \cite{dr} do not depend on the particular cone metric space either: it is enough to use the metric space $(X, \overline{d})$ instead of the cone metric space $(X, E, K, d)$.\\

\noindent{\bf (4)} It is shown in \cite{so} that cone metric spaces with a normal cone are paracompact: regarding the metrizability of cone metric spaces, this is a direct consequence of the classical result that every metric space is paracompact \cite[Theorem 20.9]{wi}.\\


\end{document}